\documentclass[12pt]{amsart}
\usepackage[T1]{fontenc}
\usepackage[latin1]{inputenc}
\usepackage{typearea}
\usepackage{geometry}
\usepackage{ulem}
\usepackage{amsmath}
\usepackage{amssymb}
\usepackage{latexsym}
\usepackage{enumerate}
\usepackage{amsthm}
\usepackage[all]{xy}
\usepackage{hhline}
\usepackage{epsf} 
\usepackage{cite}
\usepackage{verbatim}
\usepackage{mathtools}
\usepackage{comment}
\usepackage{dsfont}
\usepackage{mathrsfs}

\usepackage{color}

\newtheorem{theorem}{{\sc Theorem}}[section]
\newtheorem{lemma}[theorem]{{\sc Lemma}}
\newtheorem{prop}[theorem]{{\sc Proposition}}

\theoremstyle{remark}
\newtheorem{remark}[theorem]{{\sc Remark}}

\theoremstyle{definition}

\newcommand{\R}{\mathbb{R} }
\newcommand{\N}{\mathbb{N} }

\newcommand{\B}{\mathcal{B}}
\newcommand{\F}{\mathcal{F}}

\newcommand{\D}{\mathcal{D}}

\newcommand{\Prob}{\mathbb{P}}

\newcommand{\la}{\lambda}

\newcommand{\Om}{\Omega}

\providecommand{\babs}[1]{\bigl\lvert #1\bigr\rvert}
\providecommand{\Babs}[1]{\Bigl\lvert #1\Bigr\rvert}

\providecommand{\norm}[1]{\lVert #1\rVert}

\renewcommand{\phi}{\varphi}
\renewcommand{\epsilon}{\varepsilon}
\newcommand{\eps}{\varepsilon}
\renewcommand{\rho}{\varrho}
\renewcommand{\P}{\Prob}

\begin{document}
\title[L\'{e}vy's continuity theorem]{A short proof of L\'{e}vy's continuity theorem without using tightness}
\author{Christian D\"obler}
\thanks{\noindent Mathematisches Institut der Heinrich Heine Universit\"{a}t D\"usseldorf\\
Email: christian.doebler@hhu.de }
\begin{abstract}  
In this note we present a new short and direct proof of L\'{e}vy's continuity theorem in arbitrary dimension $d$, which does not rely on Prohorov's theorem, Helly's selection theorem or the uniqueness theorem for characteristic functions. Instead, it is based on convolution with a small (scalar) Gaussian distribution as well as on basic facts about weak convergence and measure theory. Moreover, we show how, by similar means, one may prove the fact that a distribution with integrable characteristic function is absolutely continuous with respect to $d$-dimensional Lebesgue measure and derive the formula for its density.   
\end{abstract}

\maketitle

\section{Introduction}
L\'{e}vy's continuity theorem is arguably one of the most frequently used tools for proving weak convergence of probability measures on $(\R^d,\B^d)$ ($\B^d$ being the $\sigma$-field of Borel sets in $\R^d$) and, as such, is a cornerstone of classical probability theory. Since the characteristic function of sums of independent, square-integrable random variables can be easily computed and approximated by their Taylor polynomial of degree $2$, it is also the device of choice for proving the classical Lindeberg-Feller CLT in most probability textbooks (see e.g. \cite{Chung, Bill, Dur}) and in lecture courses about measure-theoretic probability. For later reference we recall here its statement.

\begin{theorem}[L\'{e}vy's continuity theorem]\label{levytheo}
Let $\mu, \mu_n$, $n\in\N$, be probability measures on $(\R^d,\B(\R^d))$ with corresponding characteristic functions $\chi$ and $\chi_n$, $n\in\N$, respectively. Then, the following are equivalent:
\begin{enumerate}[{\normalfont (i)}]
 \item The sequence $(\mu_n)_{n\in\N}$ converges weakly to $\mu$.
 \item $\lim_{n\to\infty}\chi_n(t)=\chi(t)$ for all $t\in\R^d$.
 \end{enumerate}
\end{theorem}

Recall that the standard proof of the non-trivial implication (ii) $\Rightarrow$ (i) for $d=1$ roughly proceeds as follows:
\begin{enumerate}[1)]
 \item By means of an integral inequality involving characteristic functions (see e.g. \cite[Display (26.22)]{Bill}), (ii) implies \textit{tightness} of the sequence $(\mu_n)_{n\in\N}$.
 \item By \textit{Prohorov's theorem} and 1), each subsequence of $(\mu_n)_{n\in\N}$ admits a further subsequence converging weakly to some probability measure $\nu$.
\item If a subsequence 
 of $(\mu_n)_{n\in\N}$ converges weakly to some probability measure 
 $\nu$, then (ii) and the trivial implication (i) $\Rightarrow$ (ii) together imply that $\nu$ has characteristic function $\chi$. 
 \item By the \textit{uniqueness theorem} for characteristic functions, it holds that $\nu=\mu$.
  \item An indirect argument then completes the proof that $(\mu_n)_{n\in\N}$ converges weakly to $\mu$.  
\end{enumerate}

Consequently, for this chain of arguments, both the uniqueness theorem and Prohorov's theorem need to be established before giving the proof of Theorem \ref{levytheo}. In particular, it is necessary to introduce the concept of tightness beforehand. Moreover, the case of dimension $d$ larger than one is usually reduced to the case $d=1$ by observing that tightness is equivalent to tightness in each coordinate. \smallskip\\
In a one-semester lecture course (first course) about measure-theoretic probability, the instructor usually suffers from great time restrictions, because various topics, like product spaces, independence, several modes of convergence, strong laws of large numbers, central limit theorems, conditional expectations and probabilities, martingales and often, in the very beginning, even the fundamentals of measure and integration theory must be covered. In particular, it may be desirable to restrict the theory about weak convergence on $\R^d$ to the bare minimum that is necessary for applications within the course, like the \textit{portmanteau theorem} (including some of its consequences) and, of course, L\'{e}vy's continuity theorem.

 Since the abstract metric space version of Prohorov's theorem is usually not needed in such a course and since its proof is quite technical and lengthy, the lecturer usually rather proves \textit{Helly's selection theorem}, the distribution function version of it, instead. Moreover, as multivariate distribution functions are less natural and harder to grasp than univariate ones, quite often only the univariate version of Helly's selection theorem is actually proved, although one wants to apply its statement for general $d$. 

In the summer term of 2021, the author had to deal with the above mentioned time issues when giving a course about measure-theoretic probability theory at Heinrich Heine University D\"usseldorf. Trying to avoid Prohorov's theorem, Helly's selection theorem and the concept of tightness altogether, I discovered a proof of Theorem \ref{levytheo}, which (apart from standard knowledge of measure theory) merely relies on the following facts and which I did not manage to find in the literature:
\begin{itemize}
 \item A standard consequence of the portmanteau theorem (see Lemma \ref{port}).
 \item \textit{Scheff\'{e}'s theorem}, which is a straight-forward consequence of the dominated convergence theorem and, hence, could be given as an exercise to the students (see Proposition \ref{scheffe}).
 \item A simple formula for the probability density function (p.d.f.) of the convolution with a centered Gaussian distribution with diagonal covariance matrix, which again is a standard exercise in applying Fubini's theorem and the formula for the characteristic function of a Gaussian distribution (see Lemma \ref{conv}).
\end{itemize}

Since the uniqueness theorem itself is an immediate consequence of Theorem \ref{levytheo} and as it is not needed for our argument, it becomes a simple corollary, whereas, as outlined above, it plays a major role in the standard proof sketched above. Usually, the uniqueness theorem is proved either via the \textit{Fourier inversion theorem} (see e.g. \cite[Theorem 26.2 and Formula (29.3)]{Bill}) or by means of the \textit{Stone-Weierstrass theorem}, both of whose proofs require some extra work. In \cite[Section 9.5]{Dudley}, an alternative proof of the uniqueness theorem for characteristic functions is given, which relies on convolving with a Gaussian distribution having a small scalar covariance matrix. This idea will also be the starting point of our proof of Theorem \ref{levytheo}. 

In addition to proving Theorem \ref{levytheo}, we will show that similar arguments 
also yield the following important result, which says that a distribution $\mu$ on $\R^d$, whose characteristic functions is integrable with respect to $d$-dimensional Lebesgue measure $\la^d$, is absolutely continuous and provides a formula for its density in terms of the characteristic function.   

\begin{theorem}\label{actheo}
Let $\mu$ be a probability measure on $(\R^d,\B^d)$ with characteristic function $\chi$. If $\chi\in L^1(\la^d)$, then $\mu$ is absolutely continuous with respect to $\la^d$ and a p.d.f. for $\mu$ 
is given by
\[g(z)=\frac{1}{(2\pi)^d}\int_{\R^d}\chi(y)\exp\bigl(-i\langle z,y\rangle \bigr)d\la^d(y),\quad z\in\R^d.\]
\end{theorem}

The rest of this note is structured as follows. In Section \ref{mtproof} we state the three auxiliary results mentioned above and give our alternative proofs of Theorems \ref{levytheo} and \ref{actheo}.

\section{Alternative proofs of Theorems \ref{levytheo} and \ref{actheo}}\label{mtproof}

In what follows, we denote by $\mathcal{N}_d(a,C)$ the $d$-dimensional normal distribution with mean vector $a\in\R^d$ and positive semidefinite covariance matrix $C$ and by $\mu\ast\nu$ we denote the convolution of the two probability measures $\mu$ and $\nu$ on $(\R^d,\B^d)$. Moreover, we let $I_d$ denote the $d\times d$ identity matrix. We first state the three auxiliary results that we will need for our proofs of Theorems \ref{levytheo} and \ref{actheo}.
  
The first of these is Lemma 9.5.2 in \cite{Dudley}. As mentioned above, starting from the formula for the characteristic function of a multivariate Gaussian distribution, it can easily be left as an exercise to the students (see \cite[page 304]{Dudley} for its short proof). 

\begin{lemma}\label{conv}
 Let $\mu$ be a probability measure on $(\R^d,\B^d)$ with characteristic function $\chi$. For $\sigma\in(0,\infty)$ let $\mu_\sigma:=\mu\ast \mathcal{N}_d(0,\sigma^2 I_d)$. Then, $\mu_\sigma$ has the following p.d.f. $g_\sigma$ with respect to $\la^d$: For $z\in\R^d$, 
\[g_\sigma(z)= \frac{1}{(2\pi)^d}\int_{\R^d}\chi(y)\exp\Bigl(-i\langle z,y\rangle -\frac{\sigma^2}{2}\langle y,y\rangle\Bigr)d\la^d(y). \]
\end{lemma}

The second auxiliary result is a simplified version of \textit{Scheff\'{e}'s theorem}, which is often applied in order to conclude a CLT from a local CLT like in the de Moivre-Laplace situation. Its proof is a standard exercise in applying the dominated convergence theorem.

\begin{prop}[Scheff\'{e}'s theorem]\label{scheffe}
 Let $g,g_n$, $n\in\N$, be probability density functions with respect to $\la^d$ such that $\lim_{n\to\infty} g_n=g$ 
 $\la^d$-a.e.. Then, for all $A\in\B^d$ one has that 
 \[\lim_{n\to\infty}\int_A\babs{g_n-g}d\la^d=0.\]
 In particular, the sequence $(g_n\la^d)_{n\in\N}$ converges in total variation and, a fortiori, weakly to the probability measure $g\la^d$.
\end{prop}

As we will make use of the following lemma, our proof of Theorem \ref{levytheo} is given in terms of random vectors and \textit{convergence in distribution}, which we denote by $\stackrel{\D}\longrightarrow$. Moreover, we denote by $\stackrel{\P}\longrightarrow$ \textit{convergence in probability}. In what follows, for mere notational reasons, we will assume that all appearing random vectors are defined on the same probability space $(\Om,\F,\P)$ (see, however, Remark \ref{portrem} below). Moreover, we denote by $\norm{\cdot}$ an arbitrary fixed norm on $\R^d$.

Our third auxiliary result is a standard consequence of the portmanteau theorem (see e.g. \cite[Theorem 3.2]{Bill2}).

\begin{lemma}\label{port}
Let $X, X_n, X_{n}^{(k)}, Y^{(k)}:(\Om,\F)\rightarrow(\R^d,\B^d)$, $n,k\in\N$, be random vectors. Suppose that, for each $k\in\N$, the sequence $(X_n^{(k)})_{n\in\N}$ converges in distribution to $Y^{(k)}$ and that the sequence $(Y^{(k)})_{k\in\N}$ converges in distribution to $X$. Moreover, suppose that for each $\eps>0$:
\begin{equation*}
\lim_{k\to\infty}\limsup_{n\to\infty}\P\bigl(\norm{X_n^{(k)}-X_n}>\eps\bigr)=0.
\end{equation*}
Then, the sequence $(X_n)_{n\in\N}$ converges in distribution to $X$.
\end{lemma}

\begin{remark}\label{portrem}
\begin{enumerate}[(a)]
\item Note that, for Lemma \ref{port} to hold, it is actually sufficient that, for each fixed $(n,k)\in\N^2$, the vectors $X_n$ and $X_n^{(k)}$, $k\in\N$, are defined on the same probability space. 
\item Lemma \ref{port} in particular yields the fact that convergence in probability implies convergence in distribution.
\end{enumerate}
\end{remark}
 
We have now listed all the ingredients for our alternative proof of Theorem \ref{levytheo}.

\begin{proof}[Proof of (ii) $\Rightarrow$ (i) in Theorem \ref{levytheo}]
Let $X,Z,X_n$, $n\in\N$, be independent random vectors on a probability space $(\Om,\F,\P)$, such that $X\sim\mu$, $X_n\sim\mu_n$, $n\in\N$, and $Z\sim \mathcal{N}_d(0,I_d)$. (Actually, it suffices to construct, for each fixed $n\in\N$, independent random vectors $X(n),Z(n),X_n(n)$ on some probability space $(\Om_n,\F_n,\P_n)$ such that $X(n)\sim\mu$, $X_n(n)\sim\mu_n$ and $Z(n)\sim \mathcal{N}_d(0,I_d)$. Hence, infinite product spaces are not necessary for this proof to be carried out).

For $n,k\in\N$ define further $Z_k:=k^{-1} Z\sim \mathcal{N}_d(0,k^{-2} I_d)$ as well as $Y^{(k)}:=X+Z_k$ and $X_n^{(k)}:=X_n+ Z_k$. 
Then, by Lemma \ref{conv}, for each $k\in\N$, $Y^{(k)}$ has p.d.f. 
 \[g_k(z)= \frac{1}{(2\pi)^d}\int_{\R^d}\chi(y)\exp\Bigl(-i\langle z,y\rangle -\frac{1}{2k^2}\langle y,y\rangle\Bigr)d\la^d(y), \quad z\in\R^d, \]
and, similarly, $X_n^{(k)}$ has p.d.f.  
\[g_{n,k}(z)= \frac{1}{(2\pi)^d}\int_{\R^d}\chi_n(y)\exp\Bigl(-i\langle z,y\rangle -\frac{1}{2k^2}\langle y,y\rangle\Bigr)d\la^d(y), \quad z\in\R^d, \]
for all $n,k\in\N$. Since
\begin{equation}\label{lp1}
\Babs{\chi_n(y)\exp\Bigl(-i\langle z,y\rangle -\frac{1}{2k^2}\langle y,y\rangle\Bigr) }\leq \exp\Bigl(-\frac{1}{2k^2}\langle y,y\rangle\Bigr),\quad y\in\R^d,
\end{equation}
holds for all $n,k\in\N$ and the right hand side of \eqref{lp1} is in $L^1(\la^d)$, from (ii) and the dominated convergence theorem we conclude that 
\[\lim_{n\to\infty} g_{n,k}(z)=g_k(z),\quad z\in\R^d,\]
 for each $k\in\N$. By Proposition \ref{scheffe} this implies for each $k\in\N$ that 
\begin{equation}\label{lp2}
X_n^{(k)}\stackrel{\D}{\longrightarrow} Y^{(k)}\text{ as }n\to\infty.
\end{equation}
Moreover, for each $\eps>0$, we have that 
\[\P(\|Y^{(k)}-X\|>\eps)=\P(\|Z_k\|>\eps)=\P(\norm{Z}>k\eps)\stackrel{k\to\infty}{\longrightarrow}0,\]
implying that $Y^{(k)}\stackrel{\P}{\longrightarrow}X$ and, a fortiori, that
\begin{equation}\label{lp3}
 Y^{(k)}\stackrel{\D}{\longrightarrow}X\text{ as }k\to\infty.
\end{equation} 
Finally, for each $\eps>0$, it holds that 
\begin{align}\label{lp4}
\lim_{k\to\infty}\limsup_{n\to\infty}\P\bigl(\norm{X_n^{(k)}-X_n}>\eps\bigr)&=\lim_{k\to\infty}\limsup_{n\to\infty}\P\bigl(\norm{Z_k}>\eps\bigr)\notag\\
&=\lim_{k\to\infty}\P\bigl(\norm{Z}>k\eps\bigr)=0,
\end{align}
so that, by virtue of \eqref{lp2}-\eqref{lp4}, Lemma \ref{port} implies that $X_n\stackrel{\D}{\longrightarrow} X$ as $n\to\infty$.  
\end{proof}

\begin{proof}[Proof of Theorem \ref{actheo}]
Let $X\sim\mu$ and $Z\sim\mathcal{N}_d(0,I_d)$ be independent random vectors on a probability space $(\Om,\F,\P)$, and, for $n\in\N$, define $Z_n:=n^{-1}Z\sim\mathcal{N}_d(0,n^{-2}I_d)$ as well as $X_n:=X+Z_n$.
By Lemma \ref{conv}, for each $n\in\N$, $X_n$ has the following p.d.f. $g_n$ with respect to $\la^d$: 
\[g_n(z)=\frac{1}{(2\pi)^d}\int_{\R^d}\chi(y)\exp\Bigl(-i\langle z,y\rangle -\frac{1}{2n^2}\langle y,y\rangle\Bigr)d\la^d(y), \quad z\in\R^d.\]
Since $\chi$ is integrable by assumption, 
\[\lim_{n\to\infty} \chi(y)\exp\Bigl(-i\langle z,y\rangle -\frac{1}{2n^2}\langle y,y\rangle\Bigr)=\chi(y)\exp\bigl(-i\langle z,y\rangle\bigr)\]
and 
\[\Babs{\chi(y)\exp\Bigl(-i\langle z,y\rangle -\frac{1}{2n^2}\langle y,y\rangle\Bigr)}\leq \babs{\chi(y)},\quad y,z\in\R^d,\,n\in\N,\]
the dominated convergence theorem implies for each $z\in\R^d$ that 
\begin{equation}\label{ac1}
\lim_{n\to\infty} g_n(z)=\frac{1}{(2\pi)^d}\int_{\R^d}\chi(y)\exp\bigl(-i\langle z,y\rangle \bigr)d\la^d(y)=:g(z).
\end{equation}
Next, we show that $g$ is in fact a p.d.f.  with respect to $\la^d$. Since $g_n\geq 0$ (at least $\la^d$-a.e.) for all $n\in\N$, also $g\geq0$ ($\la^d$-a.e.).
Moreover, by Fatou's lemma,
\begin{align}\label{ac2}
0\leq\int_{\R^d} gd\la^d&=\int_{\R^d} \lim_{n\to\infty} g_nd\la^d\leq\liminf_{n\to\infty} \int_{\R^d} g_nd\la^d=1.
\end{align}
Also, for each $n\in\N$ and $z\in\R^d$ one has that
\begin{equation}\label{ac3}
\babs{g_n(z)}\leq \frac{1}{(2\pi)^d}\int_{\R^d}\babs{\chi(y)}d\la^d(y)=:C<\infty.
\end{equation}
Next, we claim that for each $\eps>0$ there is an $R\in\N$ such that  
\begin{equation}\label{ac4}
\P(\norm{X_n}\leq R)\geq 1-\eps
\end{equation}
holds for alle $n\in\N$, where, here, $\norm{\cdot}$ denotes the maximum norm (i.e. we implicitly establish tightness of the sequence $(X_n)_{n\in\N}$ here). 
Indeed, since 
\[\lim_{N\to\infty}\P(\norm{X}\leq N)=\P(X\in\R^d)=1,\]
 there is an $S\in\N$ such that  
$\P(\norm{X}\leq S)\geq 1-\eps/2$. Moreover,
\[\lim_{n\to\infty} \P(\norm{Z_n}> S)=\lim_{n\to\infty} \P(\norm{Z}> nS)=0,\]
so that, by enlarging $S$ if necessary, we can assume that $\P(\norm{Z_n}> S)<\eps/2$ for all $n\in\N$.
Letting $R:=2S$ we finally have
\begin{align*}
\P(\norm{X_n}> R)&=\P(\norm{X+Z_n}> 2S)\leq \P(\norm{Z_n}> S)+\P(\norm{X}> S)<2\frac{\eps}{2}=\eps,
\end{align*}
for each $n\in\N$ so that \eqref{ac4} is satisfied.
From \eqref{ac1}, \eqref{ac4} and the dominated convergence theorem it then follows that 
\begin{align}\label{ac5}
\int_{[-R,R]^d} gd\la^d&=\lim_{n\to\infty} \int_{[-R,R]^d} g_nd\la^d=\lim_{n\to\infty} \P(\norm{X_n}\leq R)\geq 1-\eps
\end{align}
and, since $\eps>0$ was arbitrary, \eqref{ac5} and \eqref{ac2} together imply that 
\[\int_{\R^d} gd\la^d=1,\]
so that $g$ is indeed a p.d.f.. By Proposition \ref{scheffe} and \eqref{ac1} we have for the laws $\mathcal{L}(X_n)$ of $X_n$, $n\in\N$, that $\mathcal{L}(X_n)\stackrel{w}{\longrightarrow} g\la^d$. 
On the other hand, it follows as in the proof of Theorem \ref{levytheo} that also $\mathcal{L}(X_n)\stackrel{w}{\longrightarrow} \mu$, and from the uniqueness of weak limits we conclude that $\mu$ indeed has p.d.f. $g$, as claimed.
\end{proof}

\normalem
\bibliography{levy}{}
\bibliographystyle{alpha}
\end{document}